\documentclass[reqno,11pt]{amsart}

\usepackage{mathdots}
\usepackage{tikz}
\usepackage{tikz-cd}
\usetikzlibrary{automata}
\usepackage{amssymb}
\usepackage{amsgen}
\usepackage{amsmath}
\usepackage{amsthm}
\usepackage{mathrsfs}
\usepackage{cite}
\usepackage{amsfonts}
\usepackage{enumitem}

\newcommand{\hgt}{\mathop{\mathrm {ht}}}

\newcommand{\rad}{\mathop{\mathrm{rad}}\nolimits}

\newcommand{\J}{\mathrel{\mathscr J}} 
\newcommand{\R}{\mathrel{\mathscr R}} 
\newcommand{\eL}{\mathrel{\mathscr L}} 

\usepackage{xcolor}

\newtheorem{Thm}{Theorem}[section]
\newtheorem{Prop}[Thm]{Proposition}

\newtheorem{Lemma}[Thm]{Lemma}
{\theoremstyle{definition}
}
{\theoremstyle{remark}
}
\newtheorem{Cor}[Thm]{Corollary}
{\theoremstyle{remark}
}
{\theoremstyle{remark}
}

{\theoremstyle{remark}
}
{\theoremstyle{remark}
}
{\theoremstyle{remark}
}
{\theoremstyle{remark}
\newtheorem*{Claim*}{Claim}}

\numberwithin{equation}{section}

\title[Integral quiver presentations]{Integral quiver presentations for $\mathscr R$-trivial semigroup algebras}

\author{Benjamin Steinberg}
\address[B.~Steinberg]{%
    Department of Mathematics\\
    City College of New York\\
    Convent Avenue at 138th Street\\
    New York, New York 10031\\
    USA}
\email{bsteinberg@ccny.cuny.edu}

\thanks{The author was supported by the NSF grant DMS-2452324, a Simons Foundation Collaboration Grant, award number 849561, the Australian Research Council Grant DP230103184, and Marsden Fund Grant MFP-VUW2411.}
\date{\today}

\keywords{$\R$-trivial semigroup, quiver, path algebra, bound quiver}
\subjclass[2020]{20M25,20M30, 20M50}

\begin{document}

\begin{abstract}
Semigroups in which right divisibility is a partial order, that is, $\R$-trivial semigroups, have been studied by people in algebraic combinatorics in connection with Markov chains (especially left regular bands).  Margolis and the author showed that the quiver of the algebra of an $\R$-trivial monoid is field independent.  We prove here that the integral semigroup ring of an $\R$-trivial semigroup has a bound quiver presentation provided that it is unital.   This then gives a uniform quiver presentation over any field via extension of scalars. 
\end{abstract}
\maketitle

\section{Introduction}

Bidigare, Hanlon and Rockmore, in a seminal paper~\cite{BHR}, showed that the representation theory of hyperplane face semigroups plays a key role in the analysis of a number of naturally occurring Markov chains such as the Tsetlin library and the riffle shuffle.  Brown~\cite{Brown} observed that the crucial feature of hyperplance face semigroups that made the theory work is that they are left regular bands, i.e., semigroups in which each element is idempotent and for which right divisibility is a partial order.  Ayyer, Schilling, the author and Thi\'ery~\cite{ayyer_schilling_steinberg_thiery.2013,ayyer_schilling_steinberg_thiery.sandpile.2013} argued that much of the theory carries over to random walks on $\mathscr R$-trivial semigroups, that is, semigroups where right divisiblity is a partial order, and showed that many natural Markov chains arise as such.

Study of the representation theory of $\R$-trivial semigroups has been carried out by researchers in both semigroup theory and algebraic combinatorics~\cite{SchockerRtrivial,BergeronSaliola,DO,repbook}.  Important examples of $\mathscr R$-trivial semigroups include left regular bands~\cite{Brown1,ourmemoirs,MSS}, such as hyperplane face monoids~\cite{BHR}, and $\mathscr J$-trivial semigroups~\cite{Jtrivialpaper}, such as the $0$-Hecke monoid~\cite{Carter0Hecke,Jtrivialpaper}, the Catalan monoid~\cite{catalan} and the double Catalan monoid~\cite{DoubleCatalan}.

Every finite dimensional algebra over a splitting field is Morita equivalent to the path algebra of a quiver modulo an admissible ideal~\cite{assem,benson}.   It has been observed that for certain classes of semigroups, the quiver (and even the admissible ideal) for the semigroup algebra are independent of the field~\cite{DO,ourmemoirs,Saliola,catalan}.  This led the author to ask whether these quiver presentations can in fact be obtained over the integers.   The author showed~\cite{bandquiver} that for semigroups in which each element is an idempotent, i.e., bands, it is indeed the case that their algebras admit a presentation by a bound quiver over the integers, which then gives a uniform quiver presentation over all fields.   

In this paper, we consider integral quiver presentations for algebras of $\mathscr R$-trivial semigroups.  It is already known that the quiver of the algebra of an $\mathscr R$-trivial monoid over a field does not depend on the field~\cite{DO}.  Our main result is that the integral semigroup algebra of a finite $\R$-trivial semigroup admits a bound quiver presentation provided that the algebra is unital.  This in turn leads to a uniform quiver presentation  over every field.  We also characterize when the algebra of an $\R$-trivial semigroup is unital.  

The quiver of an $\R$-trivial semigroup algebra was obtained in~\cite{DO} using homological means. Here we follow the approach of~\cite{bandquiver}.  Every $\R$-trivial semigroup $S$ has a universal homomorphism $c\colon S\to \Lambda(S)$ to a meet semilattice.   The kernel $\mathfrak J$ of the induced map $\mathbb ZS\to \mathbb Z\Lambda(S)\cong \mathbb Z^{\Lambda(S)}$ turns out to be a nilpotent ideal.
Using the results of~\cite{bandquiver}, to get the integral bound quiver presentation, we must show that $\mathfrak J/\mathfrak J^2$ is free abelian.  A new innovation in this paper is that we use derivations to prove linear independence.   The vertices of the quiver are the elements of $\Lambda(S)$ and the arrows correspond to certain equivalence classes of equivalence relations on subsets of $S$.

The paper is organized as follows.  We begin in Section~2 with some preliminaries on $\R$-trivial semigroups that are well known, but should make the paper accessible to non-specialists in semigroup theory.   Then we proceed in Section~3 to prove some structural results about the semigroup algebra of an $\R$-trivial semigroup, including bounding the nilpotency index of $\mathfrak J$, describing when the algebra is unital and obtaining a new basis for the algebra using the structure of a certain complete set of primitive idempotents.   Section~4 is extracted from~\cite{bandquiver} and discusses bound quiver presentations over commutative ground rings and recalls the Gabriel-type~\cite{assem,benson}
 criterion for the existence of a bound quiver presentation over a principal ideal domain from~\cite{bandquiver}.  Section~5 proves the existence of integral quiver presentations for unital semigroup algebras of $\R$-trivial semigroups. 

\section{$\mathscr R$-trivial semigroups}
In this paper all semigroups are assumed to be finite.   Each element $s$ of a finite semigroup $S$ has a unique idempotent power, denoted $s^\omega$.
There are natural  preorders $\leq_{\mathscr L}$, $\leq_{\mathscr R}$ and $\leq_{\J}$, due to Green~\cite{Green}, that can be defined on any semigroup.  Denoting $S$ with an adjoined identity by $S^1$, they are given by:

\begin{itemize}
    \item $s\leq_{\mathscr L} t\iff S^1s\subseteq S^1t$;
    \item $s\leq_{\mathscr R} t\iff sS^1\subseteq tS^1$;
    \item $s\leq_{\mathscr J} t\iff S^1sS^1\subseteq S^1tS^1$.
\end{itemize}

A semigroup is $\mathscr R$-trivial if $\leq_{\mathscr R}$ is a partial order, that is, $sS^1=tS^1$ implies $s=t$.  The following characterization of $\mathscr R$-trivial semigroups is well known, but we give a proof for completeness.

\begin{Prop}\label{p:R-trivial.char}
A semigroup $S$ is $\mathscr R$-trivial if and only if it satisfies $(st)^\omega s=(st)^\omega$ for all $s,t\in S$.    
\end{Prop}
\begin{proof}
    If $S$ is $\mathscr R$-trivial and $(st)^n=(st)^\omega$, then $(st)^\omega st(st)^{n-1} = (st)^\omega$, and so $(st)^\omega s=(st)^\omega$ by $\mathscr R$-triviality.  If $S$ is not $\mathscr R$-trivial, then $aS^1=bS^1$ for some $a\neq b$.  Then $as=b$ and $bt=a$ for some $s,t\in S$.  Thus, $ast=a$, and so $a(st)^\omega=a$.  Therefore, $a(st)^\omega = a\neq b=as=a(st)^\omega s$, whence $(st)^\omega\neq (st)^\omega s$.   
\end{proof}

Notice that $s\in Se=S^1e$ if and only if $se=s$ for an idempotent $e$.
The following well-known structural properties relating $\leq_{\mathscr L}$ and $\leq_\mathscr J$ in an $\R$-trivial semigroup play a crucial role.  

\begin{Prop}\label{p:J=L for idempotents}
   Let $S$ be $\mathscr R$-trivial and $e\in S$ be an idempotent.  Then the following are equivalent for $s\in S$.
   \begin{enumerate}
       \item $e\leq_{\J} s$;
       \item $es = e$;
       \item $es^{\omega}=e$;
       \item $e\leq_{\eL} s^{\omega}$;
       \item $e\leq_{\eL} s$.
   \end{enumerate}
\end{Prop}
\begin{proof}
    If $e\leq_{\J} s$, then $e=xsy$ with $x,y\in S^1$.  Then $e=ee=exsy$ implies $e=ex=exs$, and so $es= (ex)s=e$ by $\R$-triviality.  Trivially $es=e$ implies $es^{\omega}=e$, which in turn implies $S^1e\subseteq S^1s^{\omega}\subseteq S^1s\subseteq S^1sS^1$.  This completes the proof.
\end{proof}

\begin{Prop}\label{p:L.to.idem}
    If $S$ is $\mathscr R$-trivial and $S^1s=Se$ with $e$ an idempotent, then $s$ is an idempotent.
\end{Prop}
\begin{proof}
    We have $se=s$ by assumption and $es=e$ by Proposition~\ref{p:J=L for idempotents}.  Therefore, $s=se=ses=s^2$.
\end{proof}

\begin{Lemma}\label{l:stabilize}
Let $S$ be $\mathscr R$-trivial.  Then $st=s$ if and only if $st^\omega=s$.    In particular, $s^\omega=s^\omega s$.
\end{Lemma}
\begin{proof}
If $st=s$, then trivially $st^\omega=s$.  If $st^\omega =s$, then $sS^1=st^{\omega}S^1\subseteq stS^1\subseteq sS^1$, and so $st=s$ by $\mathscr R$-triviality.   In particular, $s^\omega s^\omega =s^\omega$ implies $s^\omega s=s^\omega$.
\end{proof}

Next we show that the idempotent-generated principal left ideals form a meet semilattice $\Lambda(S)$.

\begin{Prop}\label{p:meet.sml}
Let $S$ be an $\mathscr R$-trivial semigroup.  Then the set $\Lambda(S)$ of idempotent-generated principal left ideals is closed under intersection and $c\colon S\to \Lambda(S)$ given by $c(s)=Ss^{\omega}$ is a surjective homomorphism.
\end{Prop}
\begin{proof}
To prove both statements, it suffices to establish that $Ss^\omega\cap St^\omega = S(st)^\omega$.  From $(st)^\omega\leq_{\mathscr J}s,t$, we obtain $S(st)^\omega\subseteq Ss^\omega\cap St^\omega$ by Proposition~\ref{p:J=L for idempotents}.  On the other hand, if $x\in Ss^\omega\cap St^\omega$, then $xs=x$ and $xt=x$ by Lemma~\ref{l:stabilize}.  Therefore, $x(st)=x$, and so $x(st)^\omega = x$.  Thus, we have $x\in S(st)^\omega$.    
\end{proof}

Schocker~\cite{SchockerRtrivial} studied a class of monoids he called weakly ordered, which were later observed (independently by the author and Nicolas Thi\'ery) to be precisely the $\mathscr R$-trivial monoids.  The key property of weakly ordered monoids holds for $\mathscr R$-trivial semigroups whose semigroup ring contains a right identity.

\begin{Prop}\label{p:descent}
Let $S$ be an $\mathscr R$-trivial semigroup such that $s\in sS$ for all $s\in S$.  Then there is a map $d\colon S\to \Lambda(S)$ such that $st=s$ if and only if $d(s)\leq c(t)$. \end{Prop}
\begin{proof}
    Let $s\in S$.  Then $st=s$ for some $t\in S$ by assumption.  Hence $st^\omega =s$, and so $s\in St^\omega$.  Since $\Lambda(S)$ is closed under intersection, there is a unique minimum element $d(s)\in \Lambda(S)$ with $s\in d(s)$.  Then by Lemma~\ref{l:stabilize}, we have $st=s\iff st^\omega=s\iff s\in St^\omega\iff d(s)\leq c(t)$.
    \end{proof}

    It follows immediately from the definitions   that $c(s)\subseteq S^1s\subseteq d(s)$. Moreover, by Proposition~\ref{p:L.to.idem} either $s$ is an idempotent and both containments are equalities, or $s$ is not an idempotent and both containments are strict. 

\section{$\R$-trivial semigroup algebras}
Fix a commutative ring $k$ and an $\R$-trivial semigroup $S$.  If $S$ is a left regular band, meanining it consists only of idempotents, then the semigroup algebra $kS$ is unital if and only if $S$ is connected~\cite[Theorem~4.15]{ourmemoirs}.  We extend this result to the general case here. 


Let $P$ be a poset and $p\in P$.  Then $P_{\leq p} = \{q\in P\mid q\leq p\}$.
By a celebrated result of Solomon~\cite{Burnsidealgebra}, if $L$ is a meet semilattice, then $kL\cong k^L$, the isomorphism sending $\ell\in L$ to the characteristic function $\delta_{L_{\leq \ell}}$.  Let $c\colon kS\to k\Lambda(S)$ be the linear extension of $c$ and let $\tau\colon kS\to k^{\Lambda(S)}$ be the composition of $c$ with the Solomon isomorphism: $\tau(s)=\delta_{\Lambda(S)_{\leq c(s)}}$.  Note that $\ker c=\ker \tau$.

A crucial fact for us is the following theorem, proved for monoids  over a field in~\cite{AMSV} without an explicit nilpotency bound and with a bound in~\cite{SchockerRtrivial}. 

\begin{Thm}\label{t:nil}
Let $k$ be a commutative ring, let $S$ be an $\mathscr R$-trivial semigroup satisfying $s\in sS$ for all $s\in S$, and let $c\colon kS\to k\Lambda(S)$ be the induced map.  If $\mathfrak J=\ker c$, then $\mathfrak J^{n+2}=0$ where $n$ is the length of the longest $\leq_{\R}$-chain in $S$.     
\end{Thm}
\begin{proof}
Let $\hgt(s)$ be the length of the longest chain $s_0<_{\R}\cdots <_{\R} s_n=s$.  Note that $\hgt(st)\leq \hgt(s)$ with equality if and only if $st=s$, that is, $d(s)\leq c(t)$ by Proposition~\ref{p:descent}.   Thus the set $R_m = \{s\in S\mid \hgt(s)\leq m\}$  is a right ideal, and so we have a right ideal filtration $kR_0\subseteq \cdots\subseteq kR_n=kS$.  We claim that $kR_m/kR_{m-1}$ is annihilated by $\mathfrak J$ for $0\leq m\leq n$, where we interpret $kR_{-1}=0$.  Indeed,  if $\hgt(s)=m$, then by the above discussion if $t\in S$, we have that
\[(s+kR_{m-1}) t= \begin{cases}
 s+kR_{m-1}, & \text{if}\ d(s)\leq c(t)\\ kR_{m-1}, & \text{else}    
\end{cases}\]
depends only on $c(t)$.  Therefore, this module is inflated from  $k\Lambda(S)$, i.e., is annihilated by $\mathfrak J$.  It follows that $kS\cdot \mathfrak J^{n+1}=0$, and so $\mathfrak J^{n+2}=0$.
\end{proof}

The last line of the proof shows that if $kS$ is unital, then $\mathfrak J^{n+1}=0$.

When $k$ is a field, $\ker c=\rad(kB)$ as $k\Lambda(S)\cong k^{\Lambda(S)}$ is semisimple.

The following lemma is elementary and well known; see~\cite{bandquiver}.  

\begin{Lemma}\label{l:equiv}
Let $X$ be a set and $k$ a commutative ring.  Let $\equiv$ be an equivalence relation on $X$ and  $R\subseteq {\equiv}$ a binary relation.  Let $V$ be the submodule of $kX$ spanned by all differences $x-y$ with $x\equiv y$.  Then $R$ generates $\equiv$ if and only if $V$ is spanned by all differences $x-y$ with $(x,y)\in R$.  
\end{Lemma}
%

Our next proposition is from~\cite{bandquiver}.

\begin{Prop}\label{p:left.id}
Let $S$ be a finite semigroup and $k$ a commutative ring. Let $J$ be a nilpotent two-sided ideal of $kS$ such that $kS/J$ has a left identity.   Then $kS$ has a left identity if and only  if $J=kS\cdot J$.
\end{Prop}

Of course, the dual to Proposition~\ref{p:left.id} for right identities holds.  We shall apply the proposition and its dual to $\ker c$ since $k\Lambda(S)\cong k^{\Lambda(S)}$ is unital.

Call an $\mathscr R$-trivial semigroup $S$ \emph{connected} if $s\in sS$ for all $s\in S$ and if the equivalence relation $\equiv$ given by $s\equiv t$ if $c(s)=c(t)$ is generated by all pairs $(s,t)$ such that $c(s)=c(t)$ and there is $x\in S$ with $xs=s$ and $xt=t$.

 Fix once and for all, for each $X\in \Lambda(S)$, an idempotent $f_X$ with $X=Sf_X$. 

\begin{Thm}\label{t:right.connected}
Let $S$ be an $\R$-trivial semigroup.  Then the following are equivalent.
\begin{enumerate}
    \item $S$ is connected.
    \item $\mathbb ZS$ has an identity.
    \item $kS$ is unital for all commutative rings $k$.
\end{enumerate}
\end{Thm}
\begin{proof}
Trivially, (3) implies (2).  Suppose that $\mathbb ZS$ has an identity.   We prove that $S$ is  connected. First observe that $s\in sS\cap Ss$ for all $s\in S$.  Indeed, if $1=\sum_{t\in S}c_tt$, then $\sum_{t\in S}c_tts=s=\sum_{t\in S}c_tst$ implies there are $t,u\in S$ with $s=ts=su$.   Letting $\mathfrak J=\ker c$, we trivially have that $\mathbb ZS\cdot \mathfrak J = \mathfrak J$. 
Thus $\mathfrak J$ is generated by all differences $s(t-u)=st-su$ with $s,t,u\in S$ and $c(t)=c(u)$. Then $c(st)=c(su)$ and if $xs=s$ with $x\in S$, then $xsu= su$ and $xst=st$.
  It follows from Lemma~\ref{l:equiv} that the equivalence relation induced by $c$ is generated by all pairs $(x,x')$ such that $c(x)=c(x')$, and $yx=x$ and $yx'=x'$ for some $y\in S$.  Therefore, $S$ is connected.

Assume (1). Note that 
$\mathfrak J$ is nilpotent by Theorem~\ref{t:nil} and $kS/\mathfrak J\cong k\Lambda(S)\cong k^{\Lambda(S)}$ has an identity. 
We show that $kS$ has a right and left identity, and therefore a two-sided identity, using Proposition~\ref{p:left.id} and its dual.    Note that $\mathfrak J$ has basis all elements of the form $s-f_{c(s)}$ with $s\neq f_{c(s)}$.   To obtain a right identity, it suffices to show that $\mathfrak J= \mathfrak J\cdot kS$ by the dual of Proposition~\ref{p:left.id}.  Let $s\in S$ and let $st=s$ with $t\in S$ by connectedness.  Then $s^\omega t =s^\omega$ and $Ss^{\omega} = c(s) = Sf_{c(s)}$.  It follows that $(s-f_{c(s)})t = (s-f_{c(s)}s^\omega)t = s-f_{c(s)}s^\omega=s-f_{c(s)}$.  We conclude that $\mathfrak J\cdot kS=\mathfrak J$, and so $kS$ has a right identity. 

 By Proposition~\ref{p:left.id}, to obtain a left identity it suffices to show that $\mathfrak J=kS\cdot \mathfrak J$.  By our connectivity assumption, the equivalence relation induced by $c$ is generated by all pairs $(x,x')$ such that $c(x)=c(x')$ and there exists $y\in S$ with $yx=x$ and $yx'=x'$.  Therefore, by Lemma~\ref{l:equiv}, $\mathfrak J$ is spanned by all differences $x-x'$ such that $yx=x$, $yx'=x'$ for some $y\in S$ and  $c(x)=c(x')$.  Then $x-x'=y(x-x')\in kS\cdot \mathfrak J$.  We conclude that $\mathfrak J=kS\cdot \mathfrak J$, and hence $kS$ has a left identity.  This completes the proof.
\end{proof}

The author first learned of the condition for which $\mathbb ZS$ has a right identity from Itamar Stein.  The condition for a two-sided identity is new.

From here on out, we assume that $S$ is connected.  Recall that $\tau\colon \mathbb ZS\to \mathbb Z^{\Lambda(S)}$ is the surjection with kernel $\mathfrak J$ given by $\tau(s) = \delta_{\Lambda(S)_{\leq c(s)}}$.  Since $\mathfrak J$ is nilpotent, we can find a set of pairwise orthogonal idempotents $\{e_X\in\mathbb ZS\mid X\in \Lambda(S)\}$ such that $\tau(e_X)=\delta_X$ and $\sum_{X\in \Lambda(S)}e_X=1$, cf.~\cite[Proposition 21.25]{LamBook}.  

We denote by $\chi_X\colon S\to \mathbb Z$ the multiplicative homomorphism given by \[\chi_X(s) = \begin{cases} 1, & \text{if}\ c(s)\geq X,\\ 0, & \text{else}\end{cases}\] which can be extended linearly to a homomorphism $\chi_X\colon \mathbb ZS\to \mathbb Z$.
Notice that $\tau(s)\delta_X = \chi_X(s)\delta_X$ for $s\in S$, and so $\chi_X(e_Y) = \delta_{X,Y}$ (Kronecker delta).

\begin{Prop}\label{p:support}
If $X\in \Lambda(S)$ and $e_X = \sum_{s\in S}n_ss$, then $\sum_{c(s)\geq X}n_s=1$.   \end{Prop}
\begin{proof}
Note that $1=\chi_X(e_X) = \sum_{s\in S}n_s\chi_X(s) = \sum_{c(s)\geq X} n_s$.    
\end{proof}

As a consequence, we shall obtain an alternative basis for $\mathbb ZS$.

\begin{Cor}\label{c:right.fix}
    If $s\in S$, then $se_{d(s)} = s+\sum_{t<_{\R}s} n_tt$.  Hence the elements of the form $se_{d(s)}$ form a basis for $\mathbb ZS$.
\end{Cor}
\begin{proof}
Note that $st=s$ if and only if $c(t)\geq d(s)$ by Proposition~\ref{p:descent}, and so by Proposition~\ref{p:support} we have that if $e_{d(s)} = \sum_{t\in S}n_tt$, then
\[se_{d(s)} = \sum_{c(t)\geq d(s)}n_ts+\sum_{st<_{\R} s}n_tst=s+\sum_{st<_{\R} s}n_tst,\] as required.    If we choose a total order on $S$ extending $\leq_{\R}$, then the homomorphism $s\mapsto se_{d(s)}$ has an upper triangular matrix with ones along the diagonal, and hence is invertible.  The second statement follows.
\end{proof}

This basis leads to a natural basis for $\mathfrak J$.

\begin{Cor}\label{c:rad.basis}
Let $S$ be a connected $\mathscr R$-trivial semigroup and $\mathfrak J=\ker c$.  Then $\mathfrak J$ has basis all elements $se_{d(s)}$ with $s\neq s^2$ and $(s-f_{c(s)})e_{c(s)}$ with $s=s^2$, $s\neq f_{c(s)}$.    
\end{Cor}
\begin{proof}
If $s^2\neq s$, then $d(s)\nleq c(s)$ by Proposition~\ref{p:descent}.  Therefore, $\tau(se_{d(s)}) =\tau(s)\delta_{d(s)} = \chi_{d(s)}(s)\delta_{d(s)} =0$.  Thus $se_{d(s)}\in \mathfrak J$.   If $s=s^2$, then \[\tau((s-f_{c(s)})e_{c(s)})= \chi_{c(s)}(s)\delta_{c(s)}-\chi_{c(s)}(f_{c(s)})\delta_{c(s)} = 0,\] and so $(s-f_{c(s)})e_{c(s)}\in \mathfrak J$.  These elements are clearly linearly independent by Corollary~\ref{c:right.fix}.

Suppose that $a=\sum_{s\in S}n_sse_{d(s)}\in \mathfrak J$.   Note that if $s=s^2$, then $c(s)=d(s)=Ss$.  Therefore, $\tau(se_{d(s)}) = \chi_{c(s)}(s)\delta_{c(s)} = \delta_{c(s)}$.  Thus, \[0 = \tau(a) = \sum_{s=s^2}n_s\tau(se_{d(s)}) =\sum_{X\in \Lambda(S)}\sum_{Ss=X}n_s\delta_X\] where we have used Proposition~\ref{p:J=L for idempotents}.  In particular, for each $X\in \Lambda(S)$,  we have that \[\sum_{Ss=X}n_s=0.\]  We conclude that 
\[a=\sum_{s\in S}n_se_{d(s)}-\sum_{X\in \Lambda(S)}\sum_{Ss=X}n_sf_Xe_X=\sum_{s\neq s^2} n_sse_{d(s)} +\sum_{s=s^2}n_s(s-f_{c(s)})e_{c(s)}\] as required.
\end{proof}

The situation is different for $se_X$ when $X\neq d(s)$.

\begin{Prop}\label{p:wrong.support}
Let $X\in \Lambda(S)$ and $d(s)\neq X$.  Then $se_X=\sum_{t<_{\R}s}n_tt$.    Moreover, if $s=s^2$, then $c(t)<c(s)$ for all $t$ with $n_t\neq 0$. 
\end{Prop}
\begin{proof}
Let $e_{X} = \sum_{t\in S}n_tt$.
    Observe that $0=\chi_{d(s)}(e_X)=\sum_{c(t)\geq d(s)} n_t$.  Therefore, we have that \[se_X = \sum_{c(t)\geq d(s)}n_ts+\sum_{c(t)\ngeq d(s)}n_tst = \sum_{c(t)\ngeq d(s)}n_tst. \]  But, $st<_{\R} s$ for $c(t)\ngeq d(s)$, and if $s=s^2$, then $c(st)=c(s)c(t)=d(s)c(t)<d(s)=c(s)$.  This completes the proof.
\end{proof}

\section{Quivers, path algebras and admissible ideals}
A \emph{quiver} $Q$ is a finite directed graph, with multiple edges and loops allowed. We denote the set of vertices by $Q_0$ and the set of edges by $Q_1$. If $v,w$ are vertices, then $Q_1(v,w)$ will denote the set of edges from $v$ to $w$.  

If $Q$ is a quiver and $k$ is a commutative ring with unit, the \emph{path algebra} $kQ$ has basis consisting of all directed paths in $Q$, including an empty path at each vertex $v$.  We compose paths from right to left, so that if we write $s(e)$ for the source of an edge and $t(e)$ for the target, then a path has the form $p=e_n\cdots e_2e_1$ where $t(e_i) = s(e_{i+1})$.  We put $s(p) = s(e_1)$ and $t(p)=t(e_n)$.  The multiplication in $kQ$ is then given by taking $pq$ to be the concatenation when $s(p)=t(q)$, and $pq=0$ otherwise. The identity of $kQ$ is the sum of the empty paths.

The arrow ideal $J_k$ is the ideal of $kQ$ with basis all paths of positive length.  When $k=\mathbb Z$ we just write $J$.   The exact sequence
\[0\to J\to \mathbb ZQ\to \mathbb ZQ_0\to 0\] splits over $\mathbb Z$ and hence
\[0\to k\otimes J\to kQ\to kQ_0\to 0\] is exact, and so $k\otimes J= J_k$, after identifying $k\otimes Q\cong kQ$ and $k\otimes Q_0\cong kQ_0$.  

An ideal $I$ of $kQ$ is \emph{admissible} if there exists $n\geq 2$ with $J_k^n\subseteq I\subseteq J_k^2$ and $kQ/I$ is projective over $k$.   This coincides with the standard definition when $k$ is a field~\cite{assem,benson}.  A pair $(Q,I)$ consisting of a quiver $Q$ and an admissible ideal $I$ is called a \emph{bound quiver} presenting the algebra $kQ/I$.

Suppose that $I$ is an admissible ideal of $\mathbb ZQ$.  Then \[0\to I\to \mathbb ZQ\to \mathbb ZQ/I\to 0\] is exact and $\mathbb ZQ/I$ is free abelian, so the sequence splits over $\mathbb Z$.  It follows 
\begin{equation*}\label{eq:tensor.over.k}
0\to k\otimes I\to kQ\to k\otimes \mathbb ZQ/I\to 0
\end{equation*}
is exact and $k\otimes \mathbb ZQ/I$ is a free $k$-module.  Moreover, $J_k^m = (k\otimes J)^m$. 
Therefore, if $J^n\subseteq I\subseteq J^2$, then $J_k^n\subseteq k\otimes I\subseteq J_k^2$, and so $k\otimes I$ is admissible.  This leads to the following proposition from~\cite{bandquiver}.

\begin{Prop}\label{p:admiss.ov.z}
    Let $I$ be an admissible ideal of $\mathbb ZQ$.  Then for any commutative ring $k$, one has that $k\otimes I$ is an admissible ideal of $kQ$ and $kQ/(k\otimes I)\cong k\otimes \mathbb ZQ/I$.  
\end{Prop}


Next we turn to the question of when an algebra is a path algebra modulo an admissible ideal over a principal ideal domain.

If $A=kQ/I$ with $I$ an admissible ideal, then putting $\mathfrak J=J_k/I$, we have that $\mathfrak J$ is a nilpotent ideal (as it contains a power of $J_k$), $A/\mathfrak J\cong k^{Q_0}$ and $\mathfrak J/\mathfrak J^2\cong J_k/J_k^2$ is a finitely generated free $k$-module.  The converse was shown in~\cite{bandquiver} when $k$ is a principal ideal domain.  
\begin{Thm}\label{t:gabriel}
    Let $k$ be a principal ideal domain and $A$ a $k$-algebra such that:
    \begin{enumerate}
        \item $A$ is free over $k$;
        \item there is a nilpotent ideal $\mathfrak J$ of $A$ such that 
        \begin{enumerate}
            \item $A/\mathfrak J\cong k^X$ for some set $X$;
            \item $\mathfrak J/\mathfrak J^2$ is a finitely generated free $k$-module. 
        \end{enumerate}
    \end{enumerate}
   Then $A\cong kQ/I$ for a quiver $Q$ and an admissible ideal $I$ where $Q_0=X$ and $Q_1$ is in bijection with a basis for $\mathfrak J/\mathfrak J^2$. 

   More precisely, suppose that $\{e_x\mid x\in X\}$ is a complete set of orthogonal idempotents lifting $\{\delta_x\mid x\in X\}$ and $B_{yx}\subseteq e_y\mathfrak Je_x$ maps to a basis of the finitely generated free $k$-module $e_y(\mathfrak J/\mathfrak J^2)e_x$, for $x,y\in X$.   Define a quiver $Q$ by  $Q_0=X$, $Q_1(x,y) =B_{yx}$.  Then $\psi\colon kQ\to A$ given by $\psi(\varepsilon_x) = e_x$ and $\psi(b) =b$ for $b\in B_{yx}$ is surjective with kernel an admissible ideal.  
\end{Thm}

    The quiver $Q$ in Theorem~\ref{t:gabriel} is uniquely determined up to isomorphism by $A$, as remarked in~\cite{bandquiver}.


\section{The integral quiver presentation}
Fix a connected $\mathscr R$-trivial semigroup $S$.   We continue to hold fixed the idempotents $e_X,f_X$ with $X\in \Lambda(S)$.  Since $\mathfrak J=\ker c$ is nilpotent and $\mathbb ZS/\mathfrak J\cong \mathbb Z^{\Lambda(S)}$ to obtain a bound quiver presentation of $\mathbb ZS$, we must, by Theorem~\ref{t:gabriel}, find bases for $e_Y(\mathfrak J/\mathfrak J^2)e_X$ for each $X,Y\in \Lambda(S)$, where we retain our previous notation.

\subsection{Derivations}
If $R$ is a ring and $M$ is an $R$-bimodule, then a \emph{derivation} of $M$ is an additive homomorphism $d\colon R\to M$ such that $d(rs) = rd(s)+d(r)s$.  If $R=\mathbb ZS$, then a derivation is precisely a linear extension of a map $d\colon S\to M$ satisfying $d(rs) = rd(s)+d(r)s$ for all $r,s\in S$.  We shall make use of derivations in our proof of linear independence.

If $X,Y\in \Lambda(S)$, then we can consider $\mathbb Z$ as a $\mathbb ZS$-bimodule $M_{X,Y}$ via $s\cdot n = \chi_Y(s)n$ and $n\cdot s = n\chi_X(s)$.  Note that both the left and right actions factor through $\Lambda(S)$, and so $\mathfrak J=\ker c$ annihilates $\mathbb Z$ on both sides. 

\begin{Prop}\label{p:annihilate}
    Let $d\colon S\to \mathbb Z$ be a derivation of $M_{X,Y}$.  
    \begin{enumerate}
        \item  $d(\mathfrak J^2)=0$.
        \item If $a\in \mathbb ZS$ belongs to $\ker \chi_X\cap \ker \chi_Y$, then $d(e_Yae_X) = d(a)$.
    \end{enumerate}   
\end{Prop}
\begin{proof}
    To prove (1), if $x,y\in \mathfrak J$, then $d(xy) = xd(y)+d(x)y = 0$, as $\mathfrak J$ annihilates $M_{X,Y}$ on both sides.  For (2), $d(e_Yae_X) = e_Yad(e_X)+e_Yd(a)e_X+d(e_Y)ae_X = d(a)$ by assumption on $a$.
\end{proof}

In particular, if $T\subseteq \mathfrak J$ is such that we can find, for each $t\in T$, a derivation $d_t$ of $M_{X,Y}$ with $d_t(t')=\delta_{t,t'}$ (Kronecker delta) for all $t'\in T$, then the elements $t+\mathfrak J^2$ with $t\in T$ are linearly independent in $\mathfrak J/\mathfrak J^2$.

If $Z\in \Lambda(S)$, put $I_Z=\{s\in S\mid c(s)\ngeq Z\}$.  Note that $I_Z$ is a two-sided ideal of $S$ if nonempty.
Let $\sim_{X,Y}$ be the least equivalence relation on $f_YSf_X$ such that:
\begin{enumerate}
    \item $f_Ystf_X\sim_{X,Y} f_Ysf_X$ whenever $c(s)\ngeq Y$ and $c(t)\geq X$;
    \item $s\sim_{X,Y}t$ if $s,t\in I_YI_X$.
\end{enumerate}

\begin{Prop}\label{p:derivation.from.theta}
    Let $\theta\colon f_YSf_X\to \mathbb Z$ be a mapping such that $\theta$ is constant on $\sim_{X,Y}$-classes and vanishes on $I_YI_X$.  Then $d(s) = \theta(f_Ysf_X)-s\theta(f_Yf_X)$ defines a derivation of $M_{X,Y}$.
\end{Prop}
\begin{proof}
We compute that
\begin{align*}
    d(st) &= \theta(f_Ystf_X)- st\theta(f_Yf_X)\\
    sd(t)+d(s)t &= s\theta(f_Ytf_X)-st\theta(f_Yf_X)+\theta(f_Ysf_X)t-s\theta(f_Yf_X)t
\end{align*}
and so we need to show that 
\begin{equation}\label{eq:derivation}
\theta(f_Ystf_X) = s\theta(f_Ytf_X)+\theta(f_Ysf_X)t-s\theta(f_Yf_X)t
\end{equation}
Suppose first that $c(s)\geq Y$.   Then $f_Ys=f_Y$ by Proposition~\ref{p:J=L for idempotents}.  So the left hand side of \eqref{eq:derivation} becomes $\theta(f_Ytf_X)$ and the right hand side becomes $\theta(f_Ytf_X)+\theta(f_Yf_X)t - \theta (f_Yf_X)t= \theta(f_Ytf_X)$.  

So now we suppose that $c(s)\ngeq Y$.  Then the right hand side of \eqref{eq:derivation} becomes $\theta(f_Ysf_X)t$.  If $c(t)\ngeq X$, then $f_Ystf_X\in I_YI_X$, and so $\theta(f_Ystf_X)=0=\theta(f_Ysf_X)t$.  If $c(t)\geq X$, then $f_Ystf_X\sim_{X,Y} f_Ysf_X$, and so $\theta(f_Ystf_X) = \theta(f_Ysf_X) =  \theta(f_Ysf_X)t$.   We conclude that $d$ is a derivation.
\end{proof}

\subsection{Integral bases}
Our goal is now to find a basis for $e_Y(\mathfrak J/\mathfrak J^2)e_X$ for each $X,Y\in \Lambda(S)$.  Theorem~\ref{t:gabriel} will then gives us an integral quiver presentation.  There are three cases: $X,Y$ incomparable, $Y\leq X$ and $Y>X$.  First we prove some results that are valid in all cases.

\begin{Lemma}\label{l:sandwichfyfx}
    Let $j\in \mathbb ZS$ belong to $\ker\chi_X\cap \ker\chi_Y$.  Then $e_Yje_X+\mathfrak J^2= e_Yf_Yjf_Xe_X+\mathfrak J^2$. 
\end{Lemma}
\begin{proof}
    We compute \[e_Y(j-f_Yjf_X)e_X = e_Y(1-f_Y)je_X + e_Yf_Yj(1-f_X)e_X\in \mathfrak J^2\] since, for any $Z\in \Lambda(S)$, $\tau((1-f_Z)e_Z)=0=\tau(e_Z(1-f_Z))$ and $\tau(je_X)=\chi_X(j)\delta_X=0=\chi_Y(j)\delta_Y = \tau(e_Yf_Yj)$ by assumption on $j$.
\end{proof}

\begin{Lemma}\label{l:equiv.mod.rad}
    Let $X,Y\in \Lambda(S)$.
    \begin{enumerate}
        \item If $s\in I_YI_X$, then $e_Yse_X\in \mathfrak J^2$.
        \item If $s,t\in f_YSf_X$ with $s\sim_{X,Y}t$, then $e_Yse_X+\mathfrak J^2= e_Yte_Z+\mathfrak J^2$.
    \end{enumerate}
\end{Lemma}
\begin{proof}
    For (1), write $s=ab$ with $a\in I_Y$ and $b\in I_X$.  Then $\tau(e_Ya) = \chi_Y(a)\delta_Y=0=\chi_X(b)\delta_X = \tau(be_X)$, and so $e_Yse_X = e_Yabe_X\in \mathfrak J^2$.
    For (2), note that the relation $s\equiv t$ if $e_Yse_X+\mathfrak J^2 = e_Yte_X+\mathfrak J^2$ is an equivalence relation.  Moreover, if $c(s)\ngeq Y$ and $c(t)\geq X$, then $e_Yf_Ystf_Xe_X- e_Yf_Ysf_Xe_X = e_Yf_Ys(tf_X-f_X)e_X\in \mathfrak J^2$, since $\tau(e_Yf_Ys) = \chi_Y(s)\delta_Y =0$ and $c(tf_X)=X=c(f_X)$.  Therefore, taking into account (1), ${\sim_{X,Y}}\subseteq {\equiv}$.
 \end{proof}

\subsubsection{The case $X$ and $Y$ are incomparable}  Assume that $X$ and $Y$ are incomparable.
Choose a transversal $T'$ to $f_YSf_X/{\sim_{X,Y}}$, and let $T=T'\setminus \{t_0\}$ where $t_0\in T'$ is the representative of the class containing $I_YI_X$ if it is nonempty; otherwise, $T=T'$.

\begin{Thm}
    The elements of the form $e_Yte_X +\mathfrak J^2$ with $t\in T$ form a basis for $e_Y(\mathfrak J/\mathfrak J^2)e_X$.
\end{Thm}
\begin{proof}
    First note that since $X\neq Y$, $e_Y\mathbb ZSe_X=e_Y\mathfrak Je_X$.  Hence $e_Yte_X\in e_Y\mathfrak Je_X$ for each $t\in T$.   Also, $e_Yt_0e_X+\mathfrak J^2=\mathfrak J^2$ by Lemma~\ref{l:equiv.mod.rad}.  By Corollary~\ref{c:rad.basis}, we have that $e_Y\mathfrak Je_X$ is spanned by the elements of the form $e_Y(s-f_X)e_X$ with $S^1s=Sf_X$ and $e_Yse_X$ with $d(s)=X$ and $s\neq s^2$.   In the latter case, $c(s)< d(s)=X$, and so $c(s)\ngeq Y$, as $Y$ and $X$ are incomparable.  Thus $s\in \ker\chi_X\cap \ker \chi_Y$, and so $e_Yse_X+\mathfrak J^2 = e_Yf_Ysf_Xe_X+\mathfrak J^2 = e_Yte_X+\mathfrak J^2$ where $f_Ysf_X\sim_{X,Y} t$ with $t\in T'$ by Lemmas~\ref{l:sandwichfyfx} and~\ref{l:equiv.mod.rad}.    On the other hand, by Lemma~\ref{l:sandwichfyfx}, $e_Y(s-f_X)e_X+\mathfrak J^2 = e_Yf_Y(s-f_X)f_Xe_X+\mathfrak J^2 = e_Yf_Ysf_Xe_X+e_Yf_Yf_Xe_X+\mathfrak J^2 = e_Yt_1e_X+e_Yt_2e_X+\mathfrak J^2$ with $t_1,t_2\in T'$ such that $f_Ysf_X\sim_{X,Y} t_1$ and $f_Yf_X\sim_{X,Y} t_2$.  We conclude that the elements of the form $e_Yte_X+\mathfrak J^2$ with $t\in T$ span $e_Y(\mathfrak J/\mathfrak J^2)e_X$.

    To prove linear independence, it suffices by Proposition~\ref{p:annihilate} to construct, for each $t\in T$, a derivation $d_t$ of $M_{X,Y}$ such that $d_t(e_Yt'e_X) = \delta_{t,t'}$ (Kronecker delta) for $t'\in T$.   Let $\theta_t\colon f_YSf_X\to \mathbb Z$ be given by 
    \[\theta_t(s) = \begin{cases}
        1, & \text{if}\ s\sim_{X,Y} t\\ 0, & \text{else.}
    \end{cases}\]
    Then by Proposition~\ref{p:derivation.from.theta}, $d_t(s) = \theta_t(f_Ysf_X) - s\theta_t(f_Yf_X)$ defines a derivation of $M_{X,Y}$ since $t\neq t_0$.   Notice that $f_YSf_X\subseteq \ker \chi_X\cap \ker \chi_Y$ since $X$ and $Y$ are incomparable.  Therefore, if $t'\in T$, we have by Proposition~\ref{p:annihilate} 
    \[d_t(e_Yt'e_X) = d_t(t') = \theta_t(t') - t'\theta_t(f_Yf_X)=\theta_t(t') = \delta_{t,t'}\] as required. 
\end{proof}

\subsubsection{The case $Y\leq X$}  Assume that $Y\leq X$.
Fix a transversal $T'$ to $f_YSf_X/{\sim_{X,Y}}$ and let $T=T'\setminus \{t_0,t_1\}$ where $t_0\in T'$ is the representative of the class containing $I_YI_X$ if there is such a class and $t_1\in T'$ represents the class of $f_Y =f_Yf_X$. Of course, $T=T'\setminus \{t_1\}$ if $I_YI_X=\emptyset$.

\begin{Thm}
    The elements of the form $e_Yte_X +\mathfrak J^2$ with $t\in T$ form a basis for $e_Y(\mathfrak J/\mathfrak J^2)e_X$.
\end{Thm}
\begin{proof}
We first prove that $f_Y=f_Yf_X$ is in a $\sim_{X,Y}$-class of its own.  Define $\equiv$ on $f_YSf_X$ as the equivalence relation identifying all elements not equal to $f_Y$.  Then if $c(s)\ngeq Y$, $c(t)\geq X$, we clearly have $e_Yste_X\neq f_Y\neq e_Yse_X$ by applying $c$ and if $s\in I_YI_X$, then $s\neq f_Y$.  Thus ${\sim_{X,Y}}\subseteq {\equiv}$.

Next, observe that $e_Yt_0e_X+\mathfrak J^2=\mathfrak J^2$ by Lemma~\ref{l:equiv.mod.rad}.  We first prove that $e_Y(\mathfrak J/\mathfrak J^2)e_X$ is spanned by the elements of the form $e_Yse_X+\mathfrak J^2$ with $c(s)<Y$; note that $e_Yse_X\in \mathfrak J$ if $c(s)<Y$.  By Corollary~\ref{c:rad.basis}, we have that $e_Y\mathfrak Je_X$ is spanned by the elements of the form $e_Y(s-f_X)e_X$ with $S^1s=Sf_X$ and $e_Yse_X$ with $d(s)=X$ and $s\neq s^2$.   For elements of the first sort, we have, by Lemma~\ref{l:sandwichfyfx},  $e_Y(s-f_X)e_X+\mathfrak J^2 = e_Yf_Y(s-f_X)f_Xe_X+\mathfrak J^2 =\mathfrak J^2$ since $f_Ys=f_Y=f_Yf_X$, as $Y\leq X$, by Proposition~\ref{p:J=L for idempotents}.    Next consider $e_Yse_X$ with $d(s)=X$, $s^2\neq s$.  There are two cases.  

If $c(s)\ngeq Y$, then since $c(s)<X$, we have $e_Yse_X+\mathfrak J^2 = e_Yf_Ysf_Xe_X+\mathfrak J^2$ by Lemma~\ref{l:sandwichfyfx}, and $c(f_Ysf_X)<Y$.  The case $c(s)\geq Y$ can only happen  if $Y<X$ since $c(s)<X$. In this case, $f_Y=f_Ys$ by Proposition~\ref{p:J=L for idempotents}.  Therefore, $e_Yse_X - e_Yf_Ye_X = e_Y(1-f_Y)se_X\in \mathfrak J^2$ as $e_Y(1-f_Y)\in \mathfrak J$ and $c(s)<X$ implies $se_X\in \mathfrak J$. Thus $e_Yse_X+\mathfrak J^2 = e_Yf_Ye_X+\mathfrak J^2 = \sum_{c(t)<Y}n_te_Xte_Y+\mathfrak J^2$ by Proposition~\ref{p:wrong.support}. This shows that $e_Y(\mathfrak J/\mathfrak J^2)e_X$ is spanned by the elements of the form $e_Yse_X+\mathfrak J^2$ with $c(s)<Y$.

Now if $c(s)<Y\leq X$, then $e_Yse_X+\mathfrak J^2 = e_Yf_Ysf_Xe_X+\mathfrak J^2 = e_Yte_X+\mathfrak J^2$ where $t\in T'\setminus \{t_1\}$ with $f_Ysf_X\sim_{X,Y} t$ by Lemmas~\ref{l:sandwichfyfx} and~\ref{l:equiv.mod.rad}.
We conclude the elements of the form $e_Yte_X+\mathfrak J^2$ with $t\in T$ span $e_Y(\mathfrak J/\mathfrak J^2)e_X$.

To prove linear independence, it suffices by Proposition~\ref{p:annihilate} to construct, for each $t\in T$, a derivation $d_t$ of $M_{X,Y}$ such that $d_t(e_Yt'e_X) = \delta_{t,t'}$ (Kronecker delta) for $t'\in T$.   Let $\theta_t\colon f_YSf_X\to \mathbb Z$ be given by 
    \[\theta_t(s) = \begin{cases}
        1, & \text{if}\ s\sim_{X,Y} t\\ 0, & \text{else.}
                    \end{cases}\]
Since $\theta_t(f_Yf_X) = 0=\theta_t(I_YI_X)$ (as $t\neq t_0,t_1$), Proposition~\ref{p:derivation.from.theta} tells us that $d_t(s) = \theta_t(f_Ysf_X)$ defines a derivation of $M_{X,Y}$.  If $t'\in T$, then $t'\in f_YSf_X\setminus \{f_Y\}$, and so $c(t')<Y\leq X$.  Therefore, $d_t(e_Yt'e_X) = d_t(t') = \theta_t(t') = \delta_{t,t'}$ by Proposition~\ref{p:annihilate}.  This completes the proof that the $e_Yte_X+\mathfrak J^2$ with $t\in T$ form a basis for $e_Y(\mathfrak J/\mathfrak J^2)e_X$.
\end{proof}

\subsubsection{The case $Y>X$}  Assume that $Y>X$.
Pick a transversal $T'$ to $f_YSf_X/{\sim_{X,Y}}$, and let $T=T'\setminus \{t_0,t_1\}$ where $t_0\in T'$ is the representative of the class containing $I_YI_X$  if there is such a class and $t_1\in T'$ represents the class of $f_Yf_X$. We put $T=T'\setminus \{t_1\}$ if $I_YI_X=\emptyset$.

Notice that if $s\in f_YSf_X$, then $c(s)\leq X<Y$.

\begin{Prop}\label{p:c-saturated}
If $s\sim_{X,Y} t$, then $c(s)=X=c(t)$ or $c(s),c(t)<X$.
\end{Prop}
\begin{proof}
    Let $\equiv$ be the equivalence relation on $f_YSf_X$ with $s\equiv t$ if and only if $c(s)=X=c(t)$ or $c(s),c(t)<X$.  If $c(s)\ngeq Y$ and $c(t)\geq X$, then $c(f_Ystf_X) = Y\cap c(s)\cap X = c(f_Ysf_X)$, and so $f_Ystf_X\equiv f_Ysf_X$.   If $s\in I_YI_X$, then $c(s)< X$. It follows that 
    ${\sim_{X,Y}}\subseteq {\equiv}$.
\end{proof}

\begin{Thm}
    The elements of the form $e_Y(t-t_1)e_X +\mathfrak J^2$ with $t\in T$ and $c(t)=X$ and of the form $e_Yte_X+\mathfrak J^2$ with $c(t)<X$ are a basis for $e_Y(\mathfrak J/\mathfrak J^2)e_X$.
\end{Thm}
\begin{proof}
First note that $e_Yt_0e_X+\mathfrak J^2=\mathfrak J^2$ by Lemma~\ref{l:equiv.mod.rad}.  By Corollary~\ref{c:rad.basis}, we have that $e_Y\mathfrak Je_X$ is spanned by the elements of the form $e_Y(s-f_X)e_X$ with $S^1s=Sf_X$ and $e_Yse_X$ with $d(s)=X$ and $s\neq s^2$.  If $S^1s=Sf_X$, then by Lemmas~\ref{l:sandwichfyfx} and~\ref{l:equiv.mod.rad} we have that  \[e_Y(s-f_X)e_X+\mathfrak J^2 = e_Yf_Y(s-f_X)f_Xe_X+\mathfrak J^2 = e_Y(t-t_1)e_X+\mathfrak J^2\] where $t\sim_{X,Y} f_Ysf_X$, and so $c(t)=X$ as $Y>X$.  If $d(s)=X$ and $s\neq s^2$, then $c(s)<X<Y$.  Therefore, by Lemmas~\ref{l:sandwichfyfx} and~\ref{l:equiv.mod.rad}, we have that $e_Yse_X+\mathfrak J^2 = e_Yf_Ysf_Xe_X+\mathfrak J^2 = e_Yte_X+\mathfrak J^2$ where $f_Ysf_X\sim_{X,Y} t$, and hence $c(t)<X$.  It follows that the elements of the form    $e_Y(t-t_1)e_X +\mathfrak J^2$ with $t\in T$ and $c(t)=X$ and of the form $e_Yte_X+\mathfrak J^2$ with $c(t)<X$ span  $e_Y(\mathfrak J/\mathfrak J^2)e_X$.

To prove linear independence, it suffices by Proposition~\ref{p:annihilate} to construct, for each $t\in T$, a derivation $d_t$ of $M_{X,Y}$ such that $d_t(e_Y(t'-t_1)e_X)=\delta_{t,t'}$ (Kronecker delta) if $t'\in T$ with $c(t')=X$, and $d_t(e_Yt'e_X) = \delta_{t,t'}$  for $t'\in T$ with $c(t')<X$.   Let $\theta_t\colon f_YSf_X\to \mathbb Z$ be given by 
    \[\theta_t(s) = \begin{cases}
        1, & \text{if}\ s\sim_{X,Y} t\\ 0, & \text{else.}
                    \end{cases}\]
As $\theta_t(f_Yf_X) = 0=\theta_t(I_YI_X)$ (since $t\neq t_0,t_1$),  Proposition~\ref{p:derivation.from.theta} tells us that $d_t(s) = \theta_t(f_Ysf_X)$ defines a derivation of $M_{X,Y}$. Let $t'\in T$.   If $c(t')=X$ (which equals $c(t_1)$), then by Proposition~\ref{p:annihilate} 
\[d_t(e_Y(t'-t_1)e_X)= d_t(t'-t_1) = \theta_t(t')-\theta_t(t_1) =\theta_t(t')= \delta_{t,t'}.\] If $c(t')<X<Y$, then, by Proposition~\ref{p:annihilate},  $d_t(e_Yt'e_X) = d_t(t')=\theta_t(t') = \delta_{t,t'}$.  This finishes the proof of the theorem.
\end{proof}

\subsection{The quiver presentation}
The results of the previous subsection, together with Theorem~\ref{t:gabriel} yield our main result.

\begin{Thm}
    Let $S$ be a connected $\mathscr R$-trivial semigroup.  Fix an idempotent $f_X$ with $Sf_X=X$ for all $X\in \Lambda(S)$.  For $X,Y\in \Lambda(S)$, let $Z_{X,Y}$ be the complement in $f_YSf_X/{{\sim_{X,Y}}}$ of the $\sim_{X,Y}$-class containing $I_YI_X$, if nonempty, else $Z_{X,Y}=f_YSf_X/{{\sim_{X,Y}}}$.  Let $Q(S)$ be the quiver with vertex set $\Lambda(S)$ and with 
\[|Q(S)_1(X,Y)|  =\begin{cases}
    |Z_{X,Y}|, &  \text{if $X$ and $Y$  are incomparable,}\\
    |Z_{X,Y}|-1, & \text{if $X$ and $Y$ are comparable.}
\end{cases}\]      Then $\mathbb ZS\cong \mathbb ZQ(S)/I$ for an admissible ideal $I$.
Consequently, for any field $k$, we have $kS\cong kQ(S)/k\otimes I$ with $k\otimes I$ admissible.
\end{Thm}

\bibliographystyle{abbrv}
\bibliography{standard2}

\end{document}